\documentclass[11pt,reqno]{amsart}

\makeatletter
\usepackage{amssymb}
\usepackage{amssymb,amsmath,color,comment}
\usepackage{amsbsy}
\usepackage{amsfonts}
\usepackage{bbm}
\usepackage{color}

\def\marginpar#1{\ignorespaces}

\newtheorem{thm}{Theorem}
\newtheorem{prop}[thm]{Proposition}
\newtheorem{lem}[thm]{Lemma}
\newtheorem{cor}[thm]{Corollary}

\theoremstyle{definition}

\newtheorem{rmk}[thm]{Remark}
\newtheorem{exa}{Example}

\def\eqlabel#1{\def\@currentlabel{#1}}

\def\formula#1{\def\@tempa{#1}\let\@tempb\theequation\def\theequation{%
\hbox{#1}}\def\@currentlabel{(\theequation)}$$}
\def\endformula{\leqno\hbox{(\@tempa)}$$\@ignoretrue\let\theequation\@tempb}

\def\given{\hskip5\p@\relax\vrule\@width.4\p@\hskip5\p@\relax}

\newcommand{\open}[1]{%
\par\normalfont\topsep6\p@\@plus6\p@\trivlist\item[\hskip\labelsep\itshape#1%
\@addpunct{.}]\ignorespaces}

\DeclareRobustCommand{\close}[1]{%
  \ifmmode 
  \else \leavevmode\unskip\penalty9999 \hbox{}\nobreak\hfill
  \fi
  \quad\hbox{$#1$}}

\newlength{\toskip}

\def\na {\nabla}
\def\dsp{\displaystyle}

\makeatother

\def\bthm{\begin{thm}}
\def\nthm{\end{thm}}

\def\bprop{\begin{prop}}
\def\nprop{\end{prop}}

\def\brmk{\begin{rmk}}
\def\nrmk{\end{rmk}}

\def\bexa{\begin{exa}}
\def\nexa{\end{exa}}

\def\blem{\begin{lem}}
\def\nlem{\end{lem}}

\def\bcor{\begin{cor}}
\def\ncor{\end{cor}}

\newcommand{\beqq}{\begin{equation}}
 \newcommand{\neqq}{\end{equation}}

\newcommand{\bprf}{\begin{proof}}
\newcommand{\nprf}{\end{proof}}

\makeatother

\newcommand{\ee}{\mathbb{E}}

\newcommand{\nn}{\mathbb{N}}
\newcommand{\rr}{\mathbb{R}}

\def\BB{\mathcal B}

\def\EE{\mathcal E}

\def\LL{\mathcal L}
\def\MM{\mathcal M}

\def\vep{\varepsilon}
\def\<{\langle}
\def\>{\rangle}

\newcommand{\var}{\mathop{\rm Var\,}\nolimits}
\newcommand{\ent}{\mathop{\rm Ent\,}\nolimits}

\newcommand{\Hess}{\mathop{\rm Hess\,}\nolimits}

\begin{document}

\date{\today}

\title[Poincar\'e and logarithmic Sobolev for mean fields models]{Uniform Poincar\'e and logarithmic Sobolev inequalities\\
 for mean field particles systems}

\author[A. Guillin]{\textbf{\quad {Arnaud} Guillin $^{\diamondsuit}$ \, \, }}
\address{{\bf Arnaud Guillin}. Laboratoire de Math\'ematiques Blaise Pascal, CNRS-UMR
6620, Universit\'e Clermont-Auvergne (UCA), Campus Universitaire des Cezeaux, 3 Place Vasarely, 63178 Aubi\`ere, France.}
 \email{arnaud.guillin@math.univ-bpclermont.fr}

\author[W. Liu]{\textbf{\quad {Wei} Liu $^{\clubsuit}$ \,  }}
\address{{\bf Wei LIU} School of Mathematics and Statistics, Wuhan University, Wuhan, Hubei 430072, PR China; Computational Science Hubei Key Laboratory, Wuhan University, Wuhan, Hubei 430072, PR China.} \email{wliu.math@whu.edu.cn}

\author[L. Wu]{\textbf{\quad {Liming} Wu $^{\diamondsuit}$ \, \, }}
\address{{\bf Liming Wu}. Laboratoire de Math\'ematiques Blaise Pascal, CNRS-UMR
6620, Universit\'e Clermont-Auvergne (UCA), Campus Universitaire des Cezeaux, 3 Place Vasarely, 63178 Aubi\`ere, France.}
\email{Li-Ming.Wu@math.univ-bpclermont.fr}

\author[C. Zhang]{\textbf{\quad {Chaoen} Zhang $^{\diamondsuit}$ \, \, }}
\address{{\bf {Chaoen} ZHANG}\\  Laboratoire de Math\'ematiques Blaise Pascal, CNRS-UMR
6620, Universit\'e Clermont-Auvergne (UCA), Campus Universitaire des Cezeaux, 3 Place Vasarely, 63178 Aubi\`ere, France.} \email{chaoen.zhang@uca.fr}

\maketitle

 \begin{center}

\textsc{$^{\diamondsuit}$ Universit\'e Clermont-Auvergne}
\smallskip

\textsc{$^{\clubsuit}$ Wuhan University}
\smallskip

\end{center}

\begin{abstract}
In this paper we establish some explicit and sharp estimates of the spectral gap and the log-Sobolev constant for mean field particles system, uniform in the number of particles, when the confinement potential have many local minimums. Our uniform log-Sobolev inequality, based on Zegarlinski's theorem for Gibbs measures, allows us to obtain the exponential convergence in entropy of the McKean-Vlasov equation with an explicit rate constant, generalizing the result of \cite{CMV03} by means of the displacement convexity approach, or \cite{M01,M03} by Bakry-Emery technique or the recent \cite{BGG13} by dissipation of the Wasserstein distance.
\end{abstract}

\bigskip

\bigskip

\textit{ MSC 2010 : }

\section{Introduction}
Functional inequalities such as Poincar\'e or logarithmic Sobolev inequalities have nowadays an important impact on various fields of mathematics (probability, PDE, statistics,...) due to their various properties such as convergence to equilibrium (in $L^2$ or in entropy) or concentration of measure (exponential or gaussian). We refer to the beautiful book \cite{BGL-MarkovDiffusion} for an introduction (and more) to the subject as well as bibliographical references. Let us introduce these two inequalities. Let $\mu$ be a probability measure on $\rr^d$, we say that the probability measure $\mu$ satisfies a Poincar\'e (or equivalently spectral gap) inequality with (optimal) constant $\lambda_\mu$ if for all smooth functions $f$ we have
\begin{equation}\label{PI}
(PI)\qquad\qquad \lambda_1(\mu)\, {\rm Var}_\mu(f)\le \int |\nabla f|^2d\mu,
\end{equation}
where ${\rm Var}_\mu(f):=\int f^2d\mu-\left(\int fd\mu\right)^2$ denotes the variance of $f$ wrt $\mu$
and a logarithmic Sobolev inequality with (optimal) constant $\rho_\mu$ if for all smooth functions $f$ we have
\begin{equation}\label{LSI}
(LSI)\qquad\qquad \rho_{LS}(\mu)\, {\rm Ent}_\mu(f^2)\le 2\int |\nabla f|^2d\mu,
\end{equation}
where ${\rm Ent}_\mu(f^2):=\int f^2\log(f^2/\int f^2d\mu)d\mu$ denotes the entropy of $f^2$  with respect to (w.r.t. in
short) $\mu$. A famous condition to verify those inequalities is the Bakry-Emery $\Gamma_2$ criterion which says that if $d\mu=e^{-V}dx$ on $\rr^n$,
$\Hess V\ge\kappa Id>0$, then $\lambda_1(\mu)\ge \rho_{LS}(\mu)\ge \kappa$.\\
One crucial property of these two inequalities is the tensorization (or dimension free), i.e. if $\mu$ satisfies a Poincar\'e or a logarithmic Sobolev inequality then $\mu^{\otimes N}$ satisfies the same inequality with the same constant (and thus independent of $N$) leading for example to (non asymptotic) gaussian deviation inequalities refining central limit inequalities or convergence to equilibrium independent of the number of particles. However interesting physical systems are far from being independent, so that there exists a huge literature devoted to the obtention of functional inequalities such as Poincar\'e or logarithmic Sobolev inequalities, in particular to assess convergence to equilibrium, in various dependent settings such as (discrete or continuous) spin systems \cite{SZ2,SZ1,SZ3,Zegarlinski,Z96,BH99,BH00,Y00,Y01,LedouxLSI,BB19} (see also \cite{GZ03} for a survey) or mean field models \cite{M01,M03,CGM08,EberlePTRF,EGZ16} with a particular emphasis on the dependence on the number of spins or particles.\\
We will focus our attention on mean field particles system. To this end, consider the $N (\ge 2)$ interacting particles system of  mean field type :
\beqq\label{MFS}
dX^{N}_i(t) = \sqrt{2} dB_i(t) -\nabla V(X^{N}_i(t))dt - \frac 1{N-1} \sum_{j\ne i }\nabla_x W(X^{N}_i(t),X^{N}_j(t)) dt, \ i=1,\cdots, N
\neqq
where $B_1(t),\cdots,B_N(t)$ are $N$ independent Brownian motions taking values in $\rr^d$, the confinement potential $V$ is a function on $\rr^d$ of class $C^2$, and the interaction potential $W$ is a function on $\rr^d\times\rr^d$ of class $C^2$.
Its generator $\LL^{(N)}$ is given by
\beqq\label{generators}
\aligned
\LL^{(N)} f(x_1,\cdots,x_N) &= \sum_{i=1}^{ N} \LL^{(N)}_i f(x_1,\cdots,x_N) \\
\LL^{(N)}_i f(x_1,\cdots,x_N)&:=\Delta_i f(x_1,\cdots,x_N) -\nabla_i V(x_i) \cdot \na_i f(x_1,\cdots,x_N)
\\
&\quad \ \  - \frac 1{N-1} \sum_{j\ne i}(\nabla_x W)(x_i, x_j)\cdot \nabla_i  f(x_1,\cdots,x_N)
\endaligned
\neqq
for any smooth function $f$ on $(\rr^d)^N$, where $\nabla_i$ denotes the gradient  w.r.t. $x_i$, $\Delta_i$ the Laplacian w.r.t. $x_i$, and $x\cdot y=\<x,y\>$ denotes the Euclidean inner product.

The unique invariant probability measure of (\ref{MFS}) is
\beqq\label{MFM}
 \mu^{(N)}(dx_1,\cdots,dx_N)=\frac{1}{Z_N} \exp\left\{-H_N(dx_1,\cdots,dx_N)\right\} dx_1\cdots dx_N
\neqq
where
$$
 H_N(x_1,\cdots,x_N):=\sum\limits_{i=1}^{N} V(x_i)+\frac{1}{N-1}\sum\limits_{1\le i<j\le N}W(x_i,x_j)
$$
is the Hamiltonian, and $Z_N$ is the normalization constant called {\it partition function} in statistical mechanics, which is assumed to be finite throughout the paper. Without interaction (i.e. $W=0$ or constant), $\mu^{(N)}=\alpha^{\otimes N}$ (i.e. the particles are independent), where
$$
d\alpha(x)= \frac 1C e^{-V(x)}dx, \ C=\int e^{-V(x)}dx.
$$

Our first major goal is to get uniform (in the number of particles $N$) Poincar\'e or logarithmic Sobolev inequalities for the measure $\mu^{(N)}$ under tractable conditions. Malrieu \cite{M01} used Bakry-Emery's $\Gamma_2$ technique to establish a logarithmic Sobolev inequality for the mean field case thus requiring uniform convexity assumption for $V$ and $W$. Recent techniques such as Lyapunov conditions (see \cite{BCG08,BBCG} for example) are usually inefficient to get dimension-free results. For each of these inequalities we require a uniform bound for the spectral gap or the logarithmic Sobolev constants of the one particle conditional distribution. To bypass the perturbation techniques, our main assumptions for Poincar\'e inequality will be of two sorts: for the confinement potential we will need some linear growth at infinity as well as a lipschitzian spectral gap property (see Section 2 for details) which will be sufficient to get a Poincar\'e inequality for the one particle conditional distribution, and for the interaction potential a lower bound on the ``extra diagonal"  Hessian of $W$, leading to new and sharp results. A particular emphasis will be made on Curie-Weiss model and on interaction potential of the form $W(x,y)=W_0(x-y)$. The proof will repose on some  refinement of the ideas of Ledoux \cite{LedouxLSI}. For the logarithmic Sobolev inequality we will consider a translation of Zegarlinski's condition (see \cite{Zegarlinski}) for mean field model which relies on the smallness of the product of the Lipschitzian spectral gap and of the infinite norm of the Hessian of the interaction potential.\\
One of our interest to consider logarithmic Sobolev inequality for mean field particles system is to get an exponential entropic decay for the limit non-linear McKean-Vlasov equation.
Indeed, consider the non-linear McKean-Vlasov equation with an internal potential $V:\rr^d\to\rr$ and an interaction
potential $W : \rr^d\times\rr^d\to\rr$  (between two particles) so that $W(x,y)=W(y,x)$\-:
\beqq\label{McV}
\partial_t \nu_t =\Delta \nu_t +\nabla\cdot( \nu_t \nabla V ) + \nabla\cdot( \nu_t  \nabla (W\circledast \nu_t) )
\neqq
where $(\nu_t)_{t\ge0}$ is a flow of probability measures on $\rr^d$ with $\nu_0$ given, $\nabla$ is the gradient,  $\nabla\cdot$ is the divergence, and
\beqq\label{McVa}
(W\circledast \nu) (x) = \int_{\rr^d} W(x,y)d\nu(y).
\neqq
It corresponds to the self-interacting diffusion
\beqq\label{McVd}
dX_t = \sqrt{2} dB_t - \nabla V(X_t)dt -  \nabla W\circledast \nu_t (X_t)dt
\neqq
where $\nu_t$ is the law of $X_t$. It can be seen through the propagation of chaos phenomenon (see \cite{S91} for example) that the law of $X^N_1(t)$ converges to the one of $X_t$ as the number of particles $N$ tends to infinity (for each $t>0$). Via the logarithmic Sobolev inequality for the mean field particles system and a quite technical passage to the limit, we will be able to prove entropic convergence to equilibrium for the non-linear McKean-Vlasov SDE generalizing results of \cite{CMV03,BGG13}.\\
Let us finish this introduction by the plan of the paper. In the next section, we will present our set of assumptions and the main results of the paper concerning uniform Poincar\'e or logarithmic Sobolev inequality of mean field particles system as well as exponential convergence to equilibrium for McKean-Vlasov SDE (\ref{McVd}). Section 3 presents the Lipschitzian spectral gap for conditional distribution needed in the proof of the uniform Poincar\'e inequality detailed in Section 4. The translation of Zegarlinski's condition and thus the proof of uniform logarithmic Sobolev inequality are the core of Section 5. The exponential convergence to equilibrium of McKean-Vlasov SDE is finally detailed in the last Section 6.

\section{Main results}
\subsection{Framework and main assumptions.} Throughout the paper we work in the following framework.

\begin{enumerate}

\item[(H1)] The confinement potential $V : \rr^d\to \rr$ is $C^2$-smooth, its Hessian ${\rm Hess }(V)=\na^2 V=(\partial_{x_k}\partial_{x_l} V)_{1\le k,l\le d}$ of $V$ is bounded from below and there are two positive constants $c_1,c_2$ such that
\beqq\label{SHJT1}x\cdot\nabla V(x) \ge c_1|x|^2 -c_2, \ x\in \rr^d.\neqq

\item[(H2)] The pairwise interaction potential $W : \rr^d\times \rr^d\to \rr$ is $C^2$-smooth such that its Hessian $\nabla^2 W$ is bounded and
$$
  \iint \exp\left( - [V(x) + V(y) + \lambda W(x,y)]\right) dx dy<+\infty, \ \forall \lambda>0.
$$
\item[(H3)] {\bf (Lipschitzian spectral gap condition for one particle)} the following Lipschitzian constant (for the marginal conditional distribution of one particle) is finite
\beqq\label{thm1a}
c_{Lip, m} := \frac{1}{4}\int_0^{\infty}\exp\left\{\frac{1}{4}\int_0^sb_0(u){\rm d} u\right\}s{\rm d} s <+\infty
\neqq
where $b_0(r)$ is the dissipativity rate of the drift of one particle in the system (\ref{MFS}) at distance $r>0$ :
\beqq\label{21a}
 b_0(r)=\sup\limits_{x,y,z\in\rr^d : |x-y|=r} - \langle \frac{x-y}{|x-y|}, (\nabla V(x)-\nabla V(y)) +(\nabla_x W(x,z)- \nabla_x W(y,z))\rangle.
\neqq
\end{enumerate}
This last condition, taken from \cite{Wu09jfa}, is of course reminiscent of the work of Eberle \cite{EberlePTRF,EGZ16} without the interaction potential for convergence to equilibrium in $L^1$-Wasserstein distance. However in their work the interaction potential is seen only as a perturbation.

\subsection{Uniform Poincar\'e inequality for mean-field $\mu^{(N)}$}

In the sequel we shall use the notation $\nabla_{x_i,x_j}^2 H$  for a $C^2$-function $H$ on $(\rr^d)^N$, defined by
$$\nabla_{x_i,x_j}^2 H:=(\partial_{x_{ik}x_{jl}}^2H)_{1\leq k,l\leq d}$$
where $x_i=(x_{i1},x_{i2},\cdots, x_{id})\in \rr^d$. Let
\beqq\label{BJ21}\lambda_{1,m}=\inf_{N\ge 2}\inf_{1\le i\le N} \lambda_1(\mu_i)\neqq
where $\lambda_1(\mu_i)$ is the spectral gap of the conditional distribution $\mu_i=\mu_i(dx_i|x^{\hat i})$ of $x_i$ knowing $x^{\hat i}=(x_j)_{j\ne i}$, i.e. the best constant such that the following Poincar\'e inequality
$$
\lambda_1(\mu_i) {\rm Var}_{\mu_i}(f) \le \int_{\rr^d} |\na_i f|^2 d\mu_i, \ \forall f\in C^1_b(\rr^d)
$$
holds.
\begin{thm}\label{thm1} In the framework described above, we have always
\beqq\label{BJ22}
\lambda_{1,m} \ge \frac 1{c_{Lip,m}}.
\neqq
Assume that there is some constant $ h>-\lambda_{1,m}$ such that  for any $(x_1,\cdots,x_N)\in (\rr^d)^N$,
\beqq\label{thm1b}
\frac 1{N-1} (1_{i\neq j}\nabla^2_{x,y} W(x_i,x_j) )_{1\le i, j\le N}\ge h I_{dN}
\neqq
in the order of definite nonnegativity for symmetric matrices, where $I_n$ is the identity matrix of taille $n$.
Then $\mu^{(N)}$ satisfies the following Poincar\'e inequality
\beqq\label{thm1c}
\left( \lambda_{1,m} +h \right)\var_{\mu^{(N)}} (f) \le \int_{(\rr^d)^N} |\nabla f|^2 d\mu^{(N)}, \ f\in C^1_b(\rr^{dN})
\neqq
or equivalently the spectral gap $\lambda_1(\mu^{(N)})$  of $\LL^{(N)}$ on $L^2(\mu^{(N)})$, defined as the infimum of those spectral points $\lambda>0$ of $\LL^{(N)}$ on $L^2(\mu^{(N)})$,  verifies
\beqq\label{thm1d}
\lambda_1(\mu^{(N)}) \ge \lambda_{1,m}+h\ge \frac 1{c_{Lip,m}} +h.
\neqq
\end{thm}

Its proof will be given in \S 3.

The uniform Poincar\'e inequality in Theorem \ref{thm1} gives us the following explicit correlation inequality.  For any $C^1$-function $f$ on $\rr^d$, denote $\|f\|^2_{\rm Lip}$ by its Lipschitzian norm w.r.t. the Euclidean metric on $\rr^d$.

\bcor\label{thm1-cor1} Under the conditions of Theorem \ref{thm1}, for any two bounded Lipschitzian functions $f,g$ on $\rr^d$ and $i\ne j$
\beqq\label{thm1-cor1a}
{\rm Cov}_{\mu^{(N)}} (f(x_i), g(x_j)) \le   \frac {c_{\rm Lip,m}}{(1+c_{\rm Lip,m}h)(N-1)} \left(\|f\|^2_{\rm Lip} + \|g\|^2_{\rm Lip}\right)
\neqq
where ${\rm Cov}_{\mu^{(N)}}(\cdot,\cdot)$ denotes the covariance of two functions under the probability measure $\mu^{(N)}$.  Roughly speaking, two particles $x_i$ and $x_j$ become asymptotically independent at the rate $1/N$.
\ncor

\bprf The l.h.s of (\ref{thm1-cor1a}) does not depend on $(i,j)$.
Applying the Poincar\'e inequality  to $F:=\frac 1{\sqrt{N}} \sum_{i=1}^N f(x_i)$, we have
$$\aligned
{\rm Var}_{\mu^{(N)}}(F) &= {\rm Var}_{\mu^{(N)}}(f(x_1)) + (N-1) {\rm Cov}_{\mu^{(N)}} (f(x_1), f(x_2))\\
 &\le \frac {1}{\lambda_1(\mu^{(N)})} \int |\na F|^2 d\mu^{(N)} \le \frac {1}{\lambda_1(\mu^{(N)})}\|f\|^2_{\rm Lip},
\endaligned$$
and therefore $${\rm Cov}_{\mu^{(N)}} (f(x_1), f(x_2))\le \frac{1}{(N-1)\lambda_1(\mu^{(N)})}\|f\|^2_{\rm Lip}.$$

On the other hand, by the first equality above
$$\aligned
{\rm Cov}_{\mu^{(N)}} (f(x_1), f(x_2))&=\frac{1}{N-1}\left( {\rm Var}_{\mu^{(N)}}(F)-{\rm Var}_{\mu^{(N)}}(f(x_1))\right)\\
 &\ge-\frac{1}{N-1}{\rm Var}_{\mu^{(N)}}(f(x_1))\ge  -\frac{1}{(N-1)\lambda_1(\mu^{(N)})}\|f\|^2_{\rm Lip}.
\endaligned$$
Hence we get
\beqq\label{thm1-cor1ab}
|{\rm Cov}_{\mu^{(N)}} (f(x_1), f(x_2))|\le \frac{1}{(N-1)\lambda_1(\mu^{(N)})}\|f\|^2_{\rm Lip} \le \frac {c_{\rm Lip,m}}{(1+c_{\rm Lip,m}h)(N-1)}\|f\|^2_{\rm Lip},
\neqq
where the last inequality follows by (\ref{thm1d}).

Using \eqref{thm1-cor1ab}, we obtain

$$\aligned
{\rm Cov}_{\mu^{(N)}} (f(x_1), g(x_2))&=\frac 14 \left[{\rm Cov}_{\mu^{(N)}} ((f+g)(x_1), (f+g)(x_2))- {\rm Cov}_{\mu^{(N)}} ((f-g)(x_1), (f-g)(x_2))\right]\\
&\le  \frac {c_{\rm Lip,m}}{4(1+c_{\rm Lip,m}h)(N-1)}\left(\|f+g\|^2_{\rm Lip} + \|f-g\|^2_{\rm Lip} \right) \\
&\le  \frac {c_{\rm Lip,m}}{(1+c_{\rm Lip,m}h)(N-1)} \left(\|f\|^2_{\rm Lip} + \|g\|^2_{\rm Lip}\right)
\endaligned$$
the desired (\ref{thm1-cor1a}).
\nprf

\brmk{\rm The Poincar\'e inequality (\ref{thm1c}) is sharp. In fact, let $d=1$, $V(x)=x^2/2$, $W(x,y)=\beta xy$. In that case $b_0(r)=-r$ (such $W$ does not change $b_0$),  $1/c_{Lip,m}=1=\lambda_{1,m}$. Note that $\displaystyle \lambda_0:=\min\left\{1+\beta, \ 1-\frac{\beta}{N-1}\right\}$ is the smallest eigenvalue of the symmetric matrix
\[
\frac 1{N-1} (\beta 1_{i\ne j}) + I_N=\frac 1{N-1} (\beta 1_{i\ne j}) + \lambda_{1,m} I_N.
\]
Our condition (\ref{thm1b}) for the Poincar\'e inequality becomes
$$
 \lambda_0 >0.
$$
This is necessary even for well defining $\mu^{(N)}$. And our estimate (\ref{thm1d}) says that $\lambda_1(\mu^{(N)})\ge \lambda_0$.
As the matrix of the l.h.s. above is exactly the inverse of the covariance matrix of the centered gaussian distribution $\mu^{(N)}$, its spectral gap is exactly $\lambda_0$, showing so the sharpness of this theorem.
}\nrmk

\brmk\label{thm1-rem1}
Here we give an explicit estimate of $c_{Lip,m}$ under the following assumptions. Assume there are some constants $c_V,c_1,c_W,c_2\in\rr$ and $R\ge0$ such that
  \beqq\label{DCV}\langle \nabla V(x)-\nabla V(y),x-y\rangle \geq c_V|x-y|^2 -c_1|x-y|1_{[|x-y|\leq R]}\neqq
  \beqq\label{DCW}\langle \nabla_x W(x,z)-\nabla_x W(y,z),x-y\rangle \geq c_W|x-y|^2 -c_2 |x-y|1_{[|x-y|\leq R]};\neqq
 for all $x,y\in\rr^d$, and $c_V+c_W > 0$,
then we have for any $r>0$,
\begin{eqnarray*}
  b_0(r)&=&\sup\limits_{|x-y|=r,z}\langle \frac{x-y}{|x-y|},-[(\nabla V(x)-\nabla V(y)) + (\nabla_x W(x,z)-\nabla_xW(y,z))]\rangle \\
   &\leq&  -(c_V+c_W)r +(c_1+c_2)1_{[r\leq R]}
\end{eqnarray*}
which implies that
\begin{eqnarray*}
c_{Lip,m}&\leq& \frac{1}{4}\int_0^{\infty}\exp\left\{\frac{1}{4}\int_0^s [-(c_V+c_W)u+(c_1+c_2)1_{[0, R]}(u)] {\rm d} u\right\}s{\rm d} s \\
   &\leq& \frac{1}{4}\int_0^{\infty} \exp\left\{ -\frac{1}{8}(c_V+c_W)s^2+ \frac 14(c_1+c_2)R]\right\} s{\rm d} s \\
   &=&  \frac{1}{c_V+c_W}\exp\left(\frac{1}{4}(c_1+c_2)R\right).
\end{eqnarray*}
\nrmk

\bexa\label{Curie-Weiss} {\bf (Curie-Weiss model)} Let $d=1$, $V(x)=\beta(x^4/4 -x^2/2)$, $W(x,y)=-\beta K xy$ where $\beta>0$ is the inverse temperature, $K\in \rr^*$. This model is ferromagnetic or anti-ferromagnetic according to $K>0$ or $K<0$.

For this example,  we find by elementary analysis
$$
b_0(r) = -2 V'(r/2)= -2\beta (r^3/8 - r/2), \ r>0.
$$
then
\begin{eqnarray*}
  c_{Lip,m} &=& \frac{1}{4}\int_0^{\infty}\exp\left\{\frac{\beta}{4}\int_0^s(r-\frac{r^3}{4}){\rm d} r\right\}s{\rm d} s \\
   &=& \frac{1}{4}\int_0^{\infty}\exp\left\{\frac{\beta}{4}(\frac{s^2}{2}-\frac{s^4}{16})\right\}s{\rm d} s  \\
   &=& e^{\beta/4}\int_{0}^{\infty}e^{-\beta (1/2-u)^2}{\rm d} u \le  \frac {\sqrt \pi} {\sqrt \beta} e^{\beta/4}
\end{eqnarray*}
Let $\lambda(\beta)=\frac 1{c_{Lip,m}}$. By Theorem \ref{thm1}, if there exists $h>-\lambda(\beta)$ such that
\[
-\frac{\beta K}{N-1}(1_{i\neq j})\geq hI_N
\]
then $\lambda_1(\mu^{(N)})\ge h+ \lambda(\beta)$. Note that $(1_{i\neq j})$ has two eigenvalues, $N-1$ and $-1$. Hence
\[
-\frac{\beta K}{N-1}(1_{i\neq j})\geq \begin{cases} \frac{\beta K}{N-1}I_N, \ &\text{ if } K<0,\\
-\beta K I_N,\ &\text{ if }  K>0.
\end{cases}
\]
So taking
\[
h=\begin{cases} \frac{\beta K}{N-1}, \ &\text{ if } K<0,\\
-\beta K,\ &\text{ if }  K>0
\end{cases}
\]
we get by Theorem \ref{thm1},
\beqq\label{exa1b}
\lambda_1(\mu^{(N)})\ge \begin{cases}
 \frac {\sqrt \beta} {\sqrt \pi} e^{-\beta/4}+\frac{\beta K}{N-1}, \ &\text{ if } K<0,\\
 \frac {\sqrt \beta} {\sqrt \pi} e^{-\beta/4} -\beta K,\ &\text{ if }  K>0.
\end{cases}
\neqq
\nexa
(It holds automatically if the right hand side above is $\le 0$.)

In particular in the anti-ferromagnetic case (i.e. $K<0$), for any $\vep>0$ small enough, $\lambda_1(\mu^{(N)})\ge \pi^{-1/2}\beta^{1/2}e^{-\beta/4}-\vep >0$ when the number $N$ of particles is big enough: the mean field should have no phase transition.

\begin{cor}\label{thm1-cor2} Assume that $W(x,y)=W_0(x-y)$ where $W_0 :\rr^d\to\rr$ is $C^2$, even. If

\begin{enumerate}
\item
 $\nabla V$ is dissipative at infinity in the sense of (\ref{DCV}), and

\item The Hessian matrix ${\rm Hess} W_0$ of $W_0$ is bounded from below and from above:
\beqq\label{thm1-cor2a}
c_W I_d \le {\rm Hess} W_0\le C_W I_d
\neqq
and $c_W+c_V>0$.
\end{enumerate}
Then
for all $N\ge 2$,
\beqq\label{thm1-cor2b}
\lambda_1(\mu^{(N)})\ge \lambda_{1,m} -\frac N{N-1}c^-_W - C_W
\neqq
where $c^-_W$ stands for the negative part of $c_W$.
\end{cor}

\brmk\label{BE}{\rm Let us see what the Bakry-Emery $\Gamma_2$-criterion yields. If $\na^2W_0\ge c_W I_d$ and $\na^2V\ge c_V I_d$, by following the proof of the corollary above, we have $\na^2 H \ge (c_V -\frac N{N-1}c_W^-)I_{dN}$. Thus by the Bakry-Emery $\Gamma_2$-criterion,
$$
\lambda_1(\mu^{(N)})\ge \rho_{LS}(\mu^{(N)}) \ge c_V -\frac N{N-1}c_W^-
$$
where $\rho_{LS}(\mu^{(N)})$ is the log-Sobolev constant, given in the next subsection.
}\nrmk

\brmk{\rm
We notice that if $V$ is super-convex at infinity (i.e. the minimal eigenvalue of $\na^2V(x)$ tends to $+\infty$ when $|x|\to\infty$), then $c_V$ can be taken arbitrarily large, so the condition $c_W+c_V>0$ on the lower  bound $c_W$  of $\Hess W_0$ is always satisfied. In particular, if $W_0(x)=\frac{c_W}{2}|x|^2$ with $c_W<0$ (then concave and $C_W=c_W$), the uniform Poincar\'e inequality will hold for all big $N$ by (\ref{thm1-cor2b}) since, in this case,
\[
\lambda_{1,m} -\frac N{N-1}c^-_W - C_W = \lambda_{1,m} +\frac{1}{N-1}c_W.
\]
This phenomenon, apparently strange, can be intuitively explained as follows. The confinement potential, being super-convex, pushes strongly all particles towards some bounded domain; and the interaction potential $W_0$, being concave, pushes every particle far away from others. This creates an equilibrium: the meaning of our spectral gap estimate (\ref{thm1-cor2b}) for the concave potential $W_0$.

}\nrmk

We now present an example for which some much better estimates (than those in Corollary \ref{thm1-cor2}) can be obtained.

\bexa\label{exa2}
Let $W(x,y) = W_0(x-y)$ where
\[
W_0(x)=\int_{\rr^d} e^{-\sqrt{-1}\<x, y\>}{\rm d}\nu(y)+ \frac c 2|x|^2
\]
where $\nu$ is some bounded symmetric (i.e. $\nu(-A)=\nu(A)$ for any Borel subset $A$ of $\rr^d$) positive measure on $\rr^d$ with finite second moment.
Let $\Gamma_\nu=(\int y_ky_l d\nu(y))_{1\le k,l\le d}$ be the
covariance matrix of $\nu$, and $\lambda_{\rm max}(\Gamma_\nu)$ (resp. $\lambda_{\rm min}(\Gamma_\nu)$)  its maximal (resp. minimal) eigenvalue.

In \S 4, we will show the following better result :
\beqq\label{exa2a}
\lambda_1(\mu^{(N)})\ge \lambda_{1,m} + \frac 1{N-1} \left(\min\{c, -c(N-1)\} -\lambda_{\rm max}(\Gamma_\nu)\right).
\neqq
If $c\le 0$ (then the interaction potential is concave), this implies that the spectral gap of $\mu^{(N)}$ is always uniformly lower bounded.
\nexa


\subsection{Uniform log-Sobolev inequality for the mean field $\mu^{(N)}$}
Recall that some nonnegative function $f\in L\log L(\mu)$, its entropy w.r.t. the probability measure $\mu$ is defined by
$$
\ent_\mu(f):= \int f\log f d\mu - \mu(f) \log \mu(f), \ \mu(f):=\int fd\mu.
$$

\begin{thm}\label{thm2} Assume that

\begin{enumerate}

\item for some best constant $\rho_{\rm LS,m}>0$, the conditional marginal distributions $\mu_i:=\mu_i(dx_i|x^{\hat i})$ on $\rr^d$ satisfy the log-Sobolev inequality :
    \begin{equation}\label{thm2-a1}
  \rho_{\rm LS,m}\ent_{\mu_i}(f^2)\le 2 \int |\nabla f|^2 {\rm d}\mu_i, \ f\in C^1_b(\rr^d)
\end{equation}
for all $i$ and $x^{\hat i}$ ;

\item (a translation of Zegarlinski's condition)
\beqq\label{conditionZ}
\gamma_0=c_{Lip,m}  \sup_{x,y\in\rr^d,|z|=1}|\nabla_{x,y}^2W(x,y)z|<1.
\neqq
\end{enumerate}
then $\mu^{(N)}$ satisfies
\[
\rho_{\rm LS,m}(1-\gamma_0)^{2} \ent_{\mu^{(N)}} (f^2) \le 2 \int_{(\rr^d)^N} |\nabla f|^2 d\mu^{(N)},\ f\in C^1_b((\rr^d)^N)
\]
i.e. the log-Sobolev constant of $\mu^{(N)}$ verifies
\beqq\label{BJ3}
\rho_{\rm LS}(\mu^{(N)})\ge\rho_{\rm LS,m}(1-\gamma_0)^{2}.
\neqq
\end{thm}

\brmk
In this remark we present one approach to establish the first assumption in Theorem \ref{thm2}. Suppose that $
\nabla_x^2W \geq -K_0 I_d
$ and $V$ is super-convex in the sense that for any $K>0$ there exists $R>0$ such that
\[
\nabla^2 V (x)\geq K I_d, \text{ for } |x|\geq R
\]
then $V$ can be decomposed as the sum of a uniform convex function $V_c$ and a bounded function $V_b$ such that
\[
\nabla^2 V_c\geq (K_1+K_0) I_d,
\]
therefore, thanks to Bakry-Emery criterion, the probability measure
\[
\frac{1}{\tilde{Z}}\exp\left(-V_c(x_i)-\frac{1}{N-1}\sum\limits_{j:j\neq i}W(x_i,x_j)\right){\rm d} x_i
\]
satisfies a log-Sobolev inequality with constant $K_1$. By the bounded perturbation theorem, the conditional measures $\mu_i=\mu_i(\cdot|x^{\hat i}),i=1,\cdots,N$ satisfy a log-Sobolev inequality with a uniform constant $\rho_{LS,m}\ge K_1 \exp(-(\sup V_b -\inf V_b))$ which does not depend on $i,x,N$.
\nrmk

\bexa
Let us go back to the Curie-Weiss example in dimension 1: $d=1$, $V(x)=\beta(x^4/4-x^2/2)$, $W(x,y)=-\beta Kxy$ where $\beta>0$. As given before we have
$$c_{Lip,m}\le \sqrt{\frac\pi\beta}e^{\beta/4}.$$
So that
$$\gamma_0\le c_{Lip,m}  \|\nabla_{x,y}^2W\|_\infty\le \sqrt{\pi\beta}e^{\beta/4} |K|$$
which will be smaller than 1 if $\beta$ or $K$ is sufficiently small.
\nexa

\subsection{Exponential convergence of McKean-Vlasov equation in entropy and in the Wasserstein metric $W_2$} We present now an application of
the uniform log-Sobolev inequality in Theorem \ref{thm2} to the non-linear McKean-Vlasov equation.

Recall at first the relative entropy of a probability measure $\nu$ w.r.t. the given probability measure $\mu$ on $\rr^d$:
\beqq\label{Rentropy}
H(\nu|\mu):=  \begin{cases}
\int f\log f d\mu={\rm Ent}_\mu(f),\ &\text{ if }\ \nu\ll\mu, f:=\frac{d\nu}{d\mu}\\
+\infty, &\text{ otherwise. }
\end{cases}
\neqq
The $L^p$-Wasserstein distance $W_p(\nu,\mu)$ is defined by
$$
W_p(\mu,\nu)=\inf_{(X,Y)} \left(\ee |X-Y|^p\right)^{1/p}
$$
where the infimum is taken over all couples $(X,Y)$ of random variables defined on some probability space, such that the laws of $X,Y$ are respectively $\mu,\nu$ (a such couple as well as their joint law is called a {\it coupling of $(\mu,\nu)$}). Recall that the space $\MM_1^p(\rr^d)$ of probability
measures with finite $p$-moment, equipped with $L^p$-Wasserstein distance $W_p$, is complete and separable (Villani \cite{Villani}).

The Fisher-Donsker-Varadhan's information of $\nu$ w.r.t. $\mu$ is defined by
\beqq\label{Inform1}
I(\nu|\mu):=\begin{cases} \int |\nabla \sqrt{f}|^2 d\mu, \ &\text{ if } \nu\ll\mu, \sqrt{f}:=\sqrt{\frac {d\nu}{d\mu}} \in H^1_\mu \\
+\infty, &\text{ otherwise. }
\end{cases}
\neqq
where $H^1_\mu$ is the domain of the Dirichlet form   $\EE_\mu[g]=\int |\nabla g|^2 d\mu$ (well defined if $\mu$ has $C^1$-density w.r.t. $dx$). Recall that the log-Sobolev inequality for $\mu^{(N)}$ can be rewritten in
\beqq\label{BJ5}
\rho_{LS}(\mu^{(N)}) H(\nu|\mu^{(N)})\le 2 I(\nu|\mu^{(N)}), \ \nu\in\MM_1((\rr^d)^N).
\neqq
What replaces the role of the relative entropy in interacting particle system for the nonlinear McKean-Vlasov equation is the free energy of a probability measure $\nu$ on $\rr^d$:

\beqq\label{freeE}
E_f(\nu):= \begin{cases}H(\nu|\alpha) + \dsp\frac 12 \iint W(x,y) d\nu(x)d\nu(y), \ &\text{ if }   H(\nu|\alpha)<+\infty\\
+\infty &\text{ otherwise}
\end{cases}
\neqq
or more precisely the corresponding mean field entropy
\beqq\label{H_W}
H_W(\nu):= E_f(\nu) -\inf_{\tilde{\nu}\in \MM_1(\rr^d)} E_f(\tilde{\nu}).
\neqq
And the substituter of the Fisher-Donsker-Varadhan's information is:   if $\nu =f(x)dx, \int |x|^2 d\nu(x)<+\infty$ and $\nabla f\in L^1_{loc}(\rr^d)$ in the distribution sense,
\beqq\label{I_W}
I_W(\nu):= \frac 14 \int |\frac{\nabla f(x)}{f(x)} + \nabla V(x) + (\nabla_x W\circledast \nu)(x)|^2 d\nu(x),
\neqq
and $+\infty$ otherwise. Those two objects appeared both in Carrillo-McCann-Villani \cite{CMV03}. The following result generalizes the main result of \cite{CMV03} from the convex framework to the more general non-convex case.

\bthm\label{thmHW} Assume  the uniform marginal log-Sobolev inequality, i.e. \eqref{thm2-a1} with $\rho_{LS,m}>0$, and the uniqueness condition of Zegarlinski (\ref{conditionZ}). Then

\begin{enumerate}

\item There exists a unique minimizer $\nu_\infty$ of $H_W$ over $\MM_1(\rr^d)$;

\item The following (nonlinear) log-Sobolev inequality
\beqq\label{H_WI_W}
\rho_{LS} H_W(\nu)\le 2 I_W(\nu), \ \nu\in \MM_1(\rr^d)
\neqq
holds, where
$$
\rho_{LS}:=\limsup_{N\to\infty} \rho_{LS}(\mu^{(N)}) \ge \rho_{LS,m}(1-\gamma_0)^{2}.
$$

\item The following Talagrand's transportation inequality holds
\beqq\label{T2}
\rho_{LS} W^2_2(\nu,\nu_\infty)\le 2 H_W(\nu), \ \nu\in \MM_1(\rr^d)
\neqq

\item For the solution $\nu_t$ of the McKean-Vlasov equation with the given initial distribution $\nu_0$ of finite second moment,
\beqq\label{thmHWa}
H_W(\nu_t) \le e^{-t \cdot \rho_{LS} /2} H_W(\nu_0), \ t\ge0
\neqq
and in particular
\beqq\label{thmHWbb}
W^2_2(\nu_t, \nu_\infty) \le \frac 2{\rho_{LS}}e^{-t \cdot \rho_{LS} /2} H_W(\nu_0), \ t\ge0
\neqq
\end{enumerate}
\nthm

\brmk
In the work by Carrillo-McCann-Villani \cite{CMV03}, presuming the presence of confining potential, such results were obtained in the case where $W(x,y)=W_0(x-y)$ and
\begin{enumerate}
  \item either $\nabla^2V > ||(\nabla^2W)^-||_{L^\infty}$ (in particular, $V$ is uniformly strictly convex);
  \item or $W$ is strictly convex at infinity, and both $V$ and $W$ are strictly convex (possibly degenerate at the origin).
\end{enumerate}
In particular, $V$ was required to be convex in both situations. If we consider the case in dimension one, $V(x)=\beta(x^4/4-x^2/2)$ and $W_0(x)=-\beta K x^2/2$ with $K\ge0$. Then by analogous calculations than for the Curie-Weiss model, we have $c_{Lip,m}\le \sqrt{\pi/\beta}e^{\beta(1+K)^2/4}$ so that $\gamma_0\le \sqrt{\pi\beta}Ke^{\beta(1+K)^2/4}$ and thus the conditions \eqref{thm2-a1}, \eqref{conditionZ} are verified for $\beta$ or $K$ small enough for example, cases not covered in \cite{CMV03}. Our conditions are quite comparable with the results obtained in \cite{EGZ16} but they only consider convergence in $L^1$-Wasserstein distance. Remark also that the conditions are comparable to the assumptions made in \cite{DEGZ19} to get an uniform in time propagation of chaos (but in $L^1$-Wasserstein distance) which explains why we may pass to the limit in the number of particles.
\nrmk

\section{Lipschitzian spectral gap for conditional distribution}

Notice that the conditional distribution $\mu_i(dx_i):=\mu_i(dx_i|x_j, j\ne i)$ of $x_i$ knowing $x^{\hat i}:=(x_j)_{j\ne i}$ of our mean field measure $\mu^{(N)}$ defined in (\ref{MFM}) is given by
\[
{\rm d} \mu_i(x_i)= \frac{1}{Z_i} \exp\left\{-V(x_i)-\frac{1}{N-1}\sum\limits_{j:\ j\neq i}W(x_i,x_j)\right\} {\rm d}x_i
\]
where $Z_i=Z_i(x^{\hat i})$ is the normalization factor.  Let
$$
H_i(x_i):= V(x_i) + \frac{1}{N-1}\sum\limits_{j:\ j\neq i}W(x_i,x_j)
$$
be the potential associated with $\mu_i$. The generator $\LL^{(N)}_i=\Delta_i -\na_iH_i\cdot\na_i$ given in (\ref{generators}), with $(x_j)_{j\ne i}$ fixed, is symmetric w.r.t. $\mu_i$. By the definition (\ref{21a}) of $b_0(r)$, for all $x, y\in (\rr^d)^N$,
$$\aligned
&\langle \frac{x_i-y_i}{|x_i-y_i|}, -[\nabla_i H (x)-\nabla_i H(x^{\hat i,y_i})]\rangle\\
 &=\frac 1{N-1}\sum_{j\ne i}\langle \frac{x_i-y_i}{|x_i-y_i|}, -[(\nabla V(x_i) + \na_x W(x_i,x_j))-(\nabla V(y_i) + \nabla_x W(y_i,x_j)]\rangle\\
&\le b_0(|x_i-y_i|)\endaligned
$$
where $x^{\hat i, y_i}\in (\rr^d)^N$ is given by $(x^{\hat i,y_i})_j=x_j, j\ne i$, $(x^{\hat i,y_i})_i=y_i$.
So we have
the following result (due to the third named author \cite{Wu09jfa}), which is the starting point of our investigation.

\begin{lem}\label{lem-Lip} Assume (\ref{thm1a}).
Then the Poisson operator $(-\LL_i)^{-1}$ on the Banach space $C_{\rm Lip,0}(\rr^d)$ of Lipschitzian continuous functions $f$ on $\rr^d$ with $\mu_i(f)=0$, equipped with the norm $\|f\|_{\rm Lip}$, is bounded and its norm
\beqq\label{lem-Lipa}
||(-\mathcal{L}_i)^{-1}||_{\text{Lip}}\leq c_{Lip,m}
\neqq
where $c_{Lip,m}$ is given in (\ref{thm1a}).
In particular the spectral gap $\lambda_1(\mu_i)$ of $\mathcal{L}_i$ on $L^2(\mu_i)$ satisfies
\beqq\label{lem-Lipb}
\lambda_1(\mu_i)\ge \frac{1}{c_{Lip,m}}.
\neqq
\end{lem}

\section{Uniform Poincar\'{e} inequality : proof of Theorem \ref{thm1}}
Let $V\in C^2(\rr^d)$ be the confinement potential, $U$  a $C^2$-potential of interaction on $(\mathbb{R}^d)^N$ and $H(x_1,\cdots,x_N)=\sum_{i=1}^N V(x_i) + U(x_1,\cdots, x_N)$ the Hamiltonian.
Now consider the probability measure
\[
d\mu:=\frac{1}{Z}e^{-H} dx_1\cdots dx_N
\]
where $Z=\int_{(\rr^d)^N} e^{-H(x)} dx$ is the normalization constant (called often {\it partition function}), assumed to be finite. We denote by $\mu_i=\mu(dx_i|x^{\hat i})$ the conditional distribution of $x_i$ given $x^{\hat i}:=(x_1,\cdots,x_{i-1},x_{i+1},\cdots,x_{N})$ under $\mu$. It is given by
\[
\mu_i({\rm d}x_i)=\frac{1}{Z_i}e^{-U(x)-V(x_i)}{\rm d} x_i,\ Z_i=Z_i(x^{\hat i}):=\int e^{-U(x)-V(x_i)}{\rm d} x_i<+\infty  \ \text{(assumed)}.
\]
We shall describe below conditions on the Hamiltonian $H$ such that $\mu$ satisfies a Poincar{\'e} inequality, namely for some positive constant $\rho$,
\[
\rho\int f^2{\rm d} \mu\leq \int |\nabla f|^2{\rm d} \mu
\]
for every smooth function $f\in C_b^1((\rr^d)^N)$. The largest $\rho$ is called the spectral gap of $\mu$, denoted as $\lambda_1(\mu)$.

\begin{prop}\label{Prop:Ledoux_(PI if marginal PI)} Assume that $Z=\int_{(\rr^d)^N} e^{-H(x)} dx<+\infty, Z_i(x^{\hat i})<+\infty$ for all $i, x^{\hat i}$. If
\begin{enumerate}
\item the marginal conditional distributions $\mu_i$ satisfy the uniform Poincar\'e inequality, i.e.
\begin{equation}\label{Condition on SG of conditional measures}
\lambda_{1,m}:=\inf\limits\limits_{1\leq i\leq N, x^{\hat i}\in (\rr^d)^{N-1}}\lambda_1(\mu_i)>0,
\end{equation}

\item
for some constant $h\in \rr$,
\begin{equation}\label{Condition on Hessian off diagonal}
(1_{i\neq j} \nabla_{x_i,x_j}^2 U)\geq h I_{dN},
\end{equation}
in the sense of nonnegative definiteness of symmetric matrices;

\end{enumerate}
then
\[
\lambda_1(\mu)\geq h+\lambda_{1,m}.
\]
\end{prop}

This result is essentially due to Ledoux \cite{LedouxLSI}. Indeed, in the case of $d=1$, if $\Hess(U)\geq \underline{\lambda}\mbox{I}_{N}$ and $\partial_{ii}U(x)\leq \bar{\lambda}$ for all $i$ and $x^{\hat i}$, then for every $v=(v_1,v_2,\cdots,v_N)\in \mathbb{R}^N$,
\[
\sum_{i\neq j} v_i\partial_{ij}^2 U v_j =\langle \Hess(U)v,v\rangle - \sum_i v_i^2\partial_{ii}^2U \geq (\underline{h}-\bar{h})|v|^2
\]
i.e. the assumption (\ref{Condition on Hessian off diagonal}) holds with $h=\underline{\lambda}-\overline{\lambda}$. This proposition gives $\lambda_1(\mu)\ge \lambda_{1,m}+\underline{\lambda}-\overline{\lambda}$, which is the original result of Ledoux \cite{LedouxLSI}.

For the convenience of the reader, we reproduce the beautiful proof of Ledoux \cite[Prop. 3.1]{LedouxLSI}.
\begin{proof} Of course we may and will assume that $\lambda_{1,m} +h>0$.
Let $\LL=\Delta - \nabla H \cdot \nabla$ be the symmetric generator associated with the probability measure $\mu$. By the dual description of Poincar\'e inequality \cite[Prop. 4.8.3]{BGL-MarkovDiffusion}, the conclusion above is equivalent to
\[
\int (\LL f)^2 {\rm d} \mu \geq (\lambda_{1,m}+h)\int |\nabla f|^2 {\rm d} \mu.
\]
Thanks to the Bakry-Emery's formula $\int \Gamma_2(f){\rm d}\mu = \int (\LL f)^2{\rm d}\mu$ and
$$
\Gamma_2(f)= \|\na^2 f\|^2_{\rm HS} + \<\na^2 H \nabla f , \na f\>
$$
where $||A||_{\rm HS}:=(\sum_{i,j} |a_{ij}|^2)^{1/2}$ is the Hilbert-Schmidt norm of a matrix $A=(a_{ij})$, we have
\begin{eqnarray*}
\int (\LL f)^2 {\rm d} \mu  &=& \int \left(||\nabla^2 f||_{\rm HS}^2 + \langle \na^2H \nabla f,\nabla f\rangle\right){\rm d} \mu\\
     &=& \int \left( ||\nabla^2 f||_{\rm HS}^2 + \sum_{i=1}^n \langle\Hess(V)(x_i)\nabla_{x_{i}} f,\nabla_{x_{i}} f\rangle + \langle \Hess(U)\nabla f,\nabla f\rangle \right){\rm d} \mu\\
     &\geq& \sum\limits_{1\leq i\leq N} \int \int_{\mathbb{R}^d}\left(||\nabla_{x_i}^2 f||_{\rm HS}^2 + \langle(\Hess(V)(x_i)+\nabla^2_{x_i,x_i}U)\nabla_{x_{i}} f,\nabla_{x_{i}} f\rangle\right) {\rm d} \mu_i d\mu \\
       & & +\int \sum_{i\neq j}\langle \nabla_{x_i,x_j}^2U\nabla_{x_{i}}f,\nabla_{x_{j}}f\rangle{\rm d} \mu
\end{eqnarray*}

Applying the above characterization of the Poincar\'{e} inequality  but to the conditional measures $\mu_i$, we have
\[
\int\left[||\nabla^2_{x_i}f||_{\rm HS}^2 + \langle(\Hess(V)(x_i)+\nabla^2_{x_i,x_i}U)\nabla_{x_{i}} f,\nabla_{x_{i}} f\rangle\right]{\rm d} \mu_i \geq \lambda_{1,m} \int |\nabla_{x_i} f|^2 {\rm d} \mu_i
\]
for any $i$ and any given $x^{\hat i}$. Moreover by the assumption (\ref{Condition on Hessian off diagonal}),
\[
\int \sum_{i\neq j}\langle \nabla_{x_ix_j}^2U \nabla_{x_{i}}f, \nabla_{x_{j}}f\rangle{\rm d} \mu \geq h\int |\nabla f|^2 {\rm d} \mu.
\]
This, combined with the previous inequality, yields the desired inequality.
\end{proof}

We  come back to the mean field setting.
\begin{proof}[Proof of Theorem \ref{thm1}]
We shall apply Proposition \ref{Prop:Ledoux_(PI if marginal PI)} to $\mu=\mu^{(N)}$. With the notations above, the interaction potential $U$ is then given by
\beqq\label{thm1e}
U(x)=\frac{1}{N-1}\sum\limits_{1\le i<j\le N}W(x_i,x_j)=\frac12\sum_{i=1}^N U_i(x)
\neqq
where $U_i(x)=\frac{1}{N-1}\sum_{j: j\ne i} W(x_i,x_j)$.
For $i\neq j$,
\[
\na^2_{x_{i},x_{j}} U = \frac{1}{N-1}(\nabla^2_{x,y} W)(x_i,x_j)
\]
therefore the assumption (\ref{thm1b}) implies the condition (\ref{Condition on Hessian off diagonal}) with constant $h$ in Proposition \ref{Prop:Ledoux_(PI if marginal PI)}.

On the other hand, since $\mu_i(dx_i|x^{\hat i}) = e^{-[V(x_i)+U_i(x)]}dx_i/Z_i(x^{\hat i})$ and
$$
-\<\frac{x_i-y_i}{|x_i-y_i|}, \nabla_{x_i}[V(x_i)+U_i(x)]-\nabla_{x_i}[V(y_i) + U_i(x^{\hat i,y_i}) ]\> \le b_0(|x_i-y_i|)
$$
as noted in \S 3, thanks to the assumption (\ref{thm1a}), Lemma \ref{lem-Lip} yields $\lambda_1(\mu_i)\ge 1/c_{\rm Lip,m}$.

Hence we can apply Proposition \ref{Prop:Ledoux_(PI if marginal PI)} to the invariant measure $\mu^{(N)}$, and obtain (\ref{thm1d}).
\end{proof}

\bprf[\bf Proof of Corollary \ref{thm1-cor2}] In this particular context $W(x,y)=W_0(x-y)$, for $U(x)$ given by (\ref{thm1e}),
$$
\nabla^2_{x_i,x_i}U(x) =\frac 1{N-1}\sum_{j\ne i} (\na^2W_0)(x_i-x_j);\  \na^2_{x_i,x_j} U = -\frac 1{N-1}(\na^2 W_0)(x_i-x_j) \text{ for } i\ne j
$$
i.e. $\na^2 U =-\frac 1{N-1}(A_{ij})$ where $A_{ij}=(\na^2 W_0)(x_i-x_j)$ for $i\ne j$ and $A_{ii}=-\sum_{j:j\ne i} A_{ij}$. As $A_{ij}$ is symmetric and
$A_{ij}=A_{ji}$, we have
for any $u=(u_1,\cdots, u_N)$ in $(\rr^d)^N$,
$$
\aligned
-\sum_{i,j} \<u_i, A_{ij} u_j\> &= \sum_{i\ne j}\<-u_i, A_{ij}(u_j-u_i)\>=\sum_{i\ne j}\<u_j, A_{ij}(u_j-u_i)\> \\
&=\frac 1{2}\sum_{i\ne j}\<(u_j-u_i), A_{ij} (u_j-u_i)\>\\
&\ge \frac{c_W}{2} \sum_{i\ne j} |u_j-u_i|^2=c_W \sum_{i, j} \<u_j, u_j-u_i\>\ \text{(by the previous equality with $A_{ij}=I$)}\\
&=c_W N \left(|u|^2 - N|\bar u|^2\right)=c_W N |u-\bar u|^2 \\
&\ge\begin{cases} c_W N |u|^2, \ &\text{ if } c_W\le 0.\\
0 &\text{ if } c_W>0
\end{cases}
\endaligned$$
Therefore $\na^2 U \ge -c^-_W \frac N{N-1} I_{dN}$. Obviously $\na^2_{x_i,x_i} U \le C_W I_d$. Then
$$
(1_{i\ne j} \nabla^2_{x_i,x_j}U) = \na^2 U - (1_{i=j} \na^2_{x_i,x_i} U) \ge  -\left(c^-_W \frac N{N-1} +C_W\right) I_{dN}
$$

It remains to apply Proposition \ref{Prop:Ledoux_(PI if marginal PI)} to get the desired spectral gap estimate (\ref{thm1-cor2b}).

\nprf

\bprf[Proof of (\ref{exa2a}) in Example 2] Notice that
$$\aligned
(1_{i\ne j}\nabla^2_{x_i,x_j}W(x_i,x_j)) &=\left( 1_{i\ne j} [-c I_d + \int e^{-\sqrt{-1}(x_i-x_j)\cdot y} yy^T{\rm d}\nu(y)]\right)\\
&=-c (1_{i\ne j} I_d) +  \left( \int e^{-\sqrt{-1}(x_i-x_j)\cdot y} yy^T{\rm d}\nu(y)\right) - (1_{i=j}\int yy^T{\rm d}\nu(y) )\\
&\ge c P_{\bf H} -c(N-1)P_{{\bf H}^\bot} -\lambda_{\rm max}(\Gamma_\nu)I_{dN}
\endaligned
$$
where the second expression in the second line is a positive-definite matrix, 
and $P_{\bf H}, P_{{\bf H}^\bot}$ are respectively the orthogonal projection from $(\rr^d)^N$ to ${\bf H}$ and to its orthogonal complement
${\bf H}^\bot$,
$$\aligned {\bf H}&=\{x=(x_1,\cdots,x_N) ; \bar x:=\frac  1N \sum_{i=1}^N x_i=0\}, \\
 {\bf H}^\bot&=\{x=(x_1,\cdots,x_N)\in (\rr^d)^N; x_1=x_2=\cdots=x_N\}.
 \endaligned$$
Thus we obtain from Theorem \ref{thm1}
$$
\lambda_1(\mu^{(N)})\ge \lambda_{1,m} + \frac 1{N-1} \left(\min\{c, -c(N-1)\} -\lambda_{\rm max}(\Gamma_\nu)\right)
$$
which is the desired inequality (\ref{exa2a}).
\nprf

\section{Uniform log-Sobolev inequality}

Inspired by Dobrushin's uniqueness condition for the Gibbs measures, Zegarlinski \cite[Theorem 0.1]{Zegarlinski} proved a criterion about the logarithmic Sobolev inequality for the Gibbs measure $\mu=e^{-H} dx/Z$ on $(\rr^d)^N$ in terms of the conditional marginal distributions $\mu_i=\mu(dx_i|x^{\hat i})$.

Let us introduce at first Zegarlinski's dependence coefficient $c^{\bf Z}_{ij}$ of $\mu_j$ upon $x_i$: this is the best nonnegative constant such that
\begin{equation}\label{Z-Dobrushin condition C2}
  |\nabla_i(\mu_j(f^2))^{1/2}|\leq (\mu_j(|\nabla_i f|^2))^{1/2} + c_{ij}^{\bf Z} (\mu_j(|\nabla_j f|^2))^{1/2}
  \end{equation}
for all smooth strictly positive functions $f(x_1,\cdots,x_N)$. Obviously $c_{ii}^{\bf Z}=0$.
The matrix $c^{\bf Z}:=(c_{ij}^{\bf Z})_{1\le i,j\le N}$ will be called Zegarlinski's matrix of interdependence in the sequel.

\begin{thm}\label{thmZ}{\bf (Zegarlinski \cite[Theorem 0.1]{Zegarlinski})} If
\begin{enumerate}
\item $\mu_i$ satisfies a uniform log-Sobolev inequality (LSI in short), i.e.
$$\rho_{\rm LS, m}:=\inf\limits\limits_{1\leq i\leq N, x^{\hat i}\in (\rr^d)^{N-1}}\rho_{LS}(\mu_i)>0.$$

\item  The following  \textbf{Zegarlinski's condition } is verified
\beqq\label{thmZa}
\gamma:=\sup_{1\leq i\leq N}\max\{\sum_{1\leq j\leq N} c_{ji}^{\bf Z}, \sum_{1\leq j\leq N} c_{ij}^{\bf Z}\}< 1.
\neqq
\end{enumerate}
Then the Gibbs measure $\mu$ satisfies the logarithmic Sobolev inequality
\beqq\label{thmZb}
\rho_{\rm LS,m}(1-\gamma)^2\ent_{\mu}(f^2)\leq 2\mu(|\nabla f|^2)
\neqq
for all smooth bounded functions $f$ on $(\rr^d)^N$, i.e.
$$\rho_{LS}(\mu)\ge\rho_{\rm LS,m}(1-\gamma)^2.$$
\end{thm}

Our objective is to estimate $c_{ij}^{\bf Z}$. We begin with a simple observation :

\blem\label{thmZ-lem1} If for any function $g=g(x_j)\in C_b^1(\rr^d)$ on the single particle $x_j$,
\beqq\label{thmZ-lem1a}
|\na_i \mu_j(g)|\le c_{ij} \mu_j(|\na g|),
\neqq
then $c_{ij}^{\bf Z}\le c_{ij}$.
\nlem

\bprf
For any $0<g\in C_b^1((\rr^d)^N)$, by the condition (\ref{thmZ-lem1a}), we have for all $i\ne j$,
$$\aligned
|\na_i \sqrt{\mu_j(g)}| &=\frac 1{2  \sqrt{\mu_j(g)}}[ |\mu_j(\na_i g) + (\na_{x_i} \int g(x_j, y^{\hat j})  {\rm d}\mu_j(x_j|x^{\hat j}))|_{y^{\hat j}=x^{\hat j}}|]\\
&\le \frac 1{2  \sqrt{\mu_j(g)}} \left[\mu_j(|\na_i g|) + c_{ij} \mu_j(|\na_j g|)\right].
\endaligned
$$
When $g=f^2$ with $f>0$, we have by the Cauchy-Schwarz inequality for all $i,j$,
$$\aligned
\mu_j(|\na_i g|)=2\mu_j(f |\na_i f|) \le 2\sqrt{\mu_j(f^2) \mu_j(|\na_i f|^2)}.
\endaligned
$$
Substituting it into the previous inequality we get
$$
|\na_i \sqrt{\mu_j(f^2)}| \le \sqrt{\mu_j(|\na_i f|^2)} + c_{ij}  \sqrt{\mu_j(|\na_j f|^2)}
$$
so it follows $c_{ij}^{\bf Z}\le c_{ij}$.
\nprf

\blem\label{thmZ-lem2} For the mean field Gibbs measure $\mu=\mu^{(N)}$, the interdependence coefficient $c_{ji}^{\bf Z}$ satisfies
$$
c_{ji}^{\bf Z} \le \frac 1{N-1} c_{\rm Lip, m}\|\na^2_{x,y}W\|_\infty, \ i\ne j
$$
where $c_{\rm Lip, m}$ is given by (\ref{thm1a}), $$\|\na^2_{x,y}W\|_\infty:=\sup_{x,y\in\rr^d}\sup_{z\in\rr^d, |z|=1}|\na^2_{x,y}W(x,y)z|.$$
\nlem

\bprf For any $z\in \rr^d$ with $|z|=1$ and $g=g(x_i)\in C_0^2(\rr^d)$,
\[\aligned
  \na_{x_j} \mu_i(g) &= \nabla_{x_j} \big(\int g(x_i) e^{-H(x_1,x_2,\cdots,x_N)} d x_i /\int e^{-H(x_1,x_2,\cdots,x_N)} d x_i \big)\\
   &= \frac{\int g(x_i)(-\nabla_{x_j}H)e^{-H}dx_i}{\int e^{-H}dx_i} + \frac{\int g(x_i)e^{-H}dx_i \int \nabla_{x_j}H e^{-H}dx_i}{(\int e^{-H}dx_i)^2}  \\
   &=  -\int g(x_i)\nabla_{x_j}H d \mu_i + \int g(x_i)d\mu_i \int \nabla_{x_j}H d\mu_i\\
   &= {\rm Cov}_{\mu_i}(g, -\nabla_{x_j}H)={\rm Cov}_{\mu_i}(g, -\frac 1{N-1} (\na_{y} W)(x_i,x_j))
\endaligned
\]
and so
$$\aligned
z\cdot\na_{x_j} \mu_i(g)&={\rm Cov}_{\mu_i}(g, -\frac 1{N-1} (\na_{y} W)(x_i,x_j)\cdot z)\\
&=-\frac 1{N-1} \<(-\LL_i)g, (-\LL_i)^{-1}((\na_{y} W)(\cdot,x_j)\cdot z -\mu_i((\na_{y} W)(\cdot,x_j)\cdot z) \>_{\mu_i}\\
&= -\frac 1{N-1} \int \na_i g \cdot \na_i (-\LL_i)^{-1}[(\na_{y} W)(\cdot,x_j)\cdot z-\mu_i((\na_{y} W)(\cdot,x_j)\cdot z)] d\mu_i.
\endaligned
$$
By Lemma \ref{lem-Lip},
$$
\aligned
&\|\na_i (-\LL_i)^{-1}((\na_{y} W)(\cdot,x_j)\cdot z-\mu_i((\na_{y} W)(\cdot,x_j)\cdot z))\|_{L^\infty(\mu_i)}\\
&\le c_{\rm Lip,m} \sup_{x_i,x_j} |\na_{x_i} ((\na_{y} W)(x_i,x_j)\cdot z)|\\
&= c_{\rm Lip,m} \sup_{x,y\in\rr^d}|\na^2_{x,y}W(x,y)z|\\
&\le c_{\rm Lip,m} \|\na^2_{x,y}W\|_\infty.
 \endaligned$$
Plugging it into the previous inequality, we obtain
$$
|\na_{x_j} \mu_i(g)|=\sup_{|z|=1} |z\cdot\na_{x_j} \mu_i(g)|\le \frac 1{N-1}c_{\rm Lip,m} \|\na^2_{x,y}W\|_\infty  |\mu_i(\nabla_{x_i}g)|
$$
which, by Lemma \ref{thmZ-lem1}, completes the proof.
\nprf

\bprf[Proof of Theorem \ref{thm2}]  By Lemma \ref{thmZ-lem2},
$$
\gamma=\sup_{1\leq i\leq N}\max\big\{\sum_{1\leq j\leq N} c_{ji}^{\bf Z}, \sum_{1\leq j\leq N} c_{ij}^{\bf Z}\big\} \le c_{\rm Lip,m} \|\na^2_{x,y}W\|_\infty=\gamma_0<1.
$$
Then Theorem \ref{thm2} follows directly from Theorem \ref{thmZ}.
\nprf

\section{Exponential convergence of McKean-Vlasov equation}

Assume that $\mu^{(N)}$ satisfies a uniform log-Sobolev inequality with constant
$$\dsp \rho_{LS}=\limsup_{N\to \infty}\rho_{LS}(\mu^{(N)})>0.$$ That is the case if $c_{\rm Lip,m} \|\na^2_{x,y}W\|_\infty<1$ by Theorem \ref{thm2}, more precisely
\[
\rho_{LS}\ge \rho_{LS,m}(1-c_{\rm Lip,m} \|\na^2_{x,y}W\|_\infty )^2.
\]

\subsection{Free energy, entropy related to the McKean-Vlasov equation}

The entropy $H_W(\nu)$ can be identified as the mean relative entropy per particle of $\nu^{\otimes N}$ w.r.t. the mean field Gibbs measure $\mu^{(N)}$:

\blem\label{thmHW-lem1} For any probability measure $\nu$ on $\rr^d$ such that $H(\nu|\alpha)<+\infty$,
\beqq\label{thmHW-lem1a}
\frac 1N H(\nu^{\otimes N}|\mu^{(N)})\to H_W(\nu).
\neqq
\nlem

\bprf Recall that $\alpha = \frac 1C e^{-V} dx$. By the assumption (H1), it is known that (\cite{CGW10})
\beqq\label{SJTU2}
\int e^{\lambda_0 |x|^2} d\alpha(x)<+\infty \text{ for some }\lambda_0>0.
\neqq
Let
$$
\tilde Z_N := \int \exp\left(-  \frac 1{2(N-1)} \sum_{i\ne j} W(x_i,x_j)\right) d\alpha^{\otimes N}
$$
so that
$$d \mu^{(N)}=
\frac 1{\tilde Z_N} \exp\left(-  \frac 1{2(N-1)} \sum_{i\ne j} W(x_i,x_j)\right) d\alpha^{\otimes N}.
$$
Let $\nu\in\MM_1(\rr^d)$ such that $H(\nu|\alpha)<+\infty$. Since $H(\nu^{\otimes 2}|\alpha^{\otimes 2})=2 H(\nu|\alpha)<+\infty$, by Donsker-Varadhan's variational formula of entropy,  (\ref{SJTU2}) and the fact that
$|W(x,y)|\le C(1+|x|^2+|y|^2)$ (for $\nabla^2 W$ is bounded),  we have $W\in L^1(\nu^{\otimes 2})$. Therefore

$$\aligned
\frac 1N H(\nu^{\otimes N}|\mu^{(N)})&=\frac 1N\int \log \frac{d \nu^{\otimes N}}{d\mu^{(N)}} d\nu^{\otimes N}\\
&=\frac 1N \int \sum_{i=1}^N \log \frac {d\nu}{d\alpha} (x_i) d\nu^{\otimes N} + \int \frac 1{2N(N-1)} \sum_{i\ne j} W(x_i,x_j)  d\nu^{\otimes N} +\frac 1N \log\tilde Z_N\\
&= H(\nu|\alpha) + \frac 12 \iint W(x,y) d\nu(x) d\nu(y)+\frac 1N \log \tilde Z_N
\endaligned$$
By \cite[(3.30)]{LW18},

$$
\lim_{N\to\infty}\frac 1N \log \tilde Z_N = -\inf_{\nu} E_{f}(\nu).
$$
Combining those two equalities we obtain (\ref{thmHW-lem1a}).
\nprf

The following super-additivity of the relative entropy w.r.t. a product probability measure should be known.

\blem\label{thmHW-lem2}  Let $\prod_{i=1}^N \alpha_i, Q$ be respectively a product probability measure and a probability measure on $E_1\times\cdots\times E_N$ where $E_i$'s are Polish spaces, and $Q^i$ the marginal distribution of $x_i$ under $Q$. Then
$$
H(Q| \prod_{i=1}^N \alpha_i)\ge \sum_{i=1}^N H(Q^i|\alpha_i).
$$
\nlem

\bprf Let $Q_i(\cdot|x_{[1,i-1]})$ be the conditional distribution of $x_i$ knowing $x_{[1,i-1]}=(x_1,\cdots,i-1)$ (knowing nothing if $i=1$). We have
$$\aligned
H(Q| \prod_{i=1}^N \alpha_i)&=\ee^Q \log \frac{dQ}{d \prod_{i=1}^N \alpha_i}=\ee^Q\sum_{i=1}^N \log \frac{Q_i(dx_i|x_{[1,i-1]})}{\alpha_i(dx_i)}\\
&=\ee^{Q} \sum_{i=1}^n H(Q_i(\cdot|x_{[1,i-1]})|\alpha_i).
\endaligned$$
Since $\ee^Q Q_i(\cdot|x_{[1,i-1]})=Q^i(\cdot)$, we obtain by the convexity of the relative entropy
$$
\ee^{Q} H(Q_i(\cdot|x_{[1,i-1]})|\alpha_i)\ge H(Q^i|\alpha_i)
$$
where the desired super-additivity follows.
\nprf

\blem\label{thmHW-lem3}  Let $\mu$ be a probability measure on some Polish space $S$ and $U : S\to (-\infty, +\infty]$ a measurable potential satisfying
$$
\int e^{-p U} d\mu <+\infty
$$
for some $p>1$. Consider the Boltzmann probability measure $\mu_U=e^{-U} d\mu/C$. If $H(\nu|\mu_U)<+\infty$, then $H(\nu|\mu)<+\infty$ and $U\in L^1(\nu)$, and
$$
H(\nu|\mu_U)=H(\nu|\mu) + \int U d\nu - \log \int e^{-U} d\mu.
$$
\nlem

\bprf For any measurable function $f$ on $S$, let
$$
\Lambda_\mu(f):=\log \int e^f d\mu\in (-\infty,+\infty]
$$
be the log-Laplace transform w.r.t. $\mu$, which is convex in $f$ (by H\"older's inequality). Then
$$
\Lambda_{\mu_U}(f) =\log \int e^f d\mu_U=\Lambda_\mu(-U+f) -\Lambda_\mu(-U) \le \frac 1p \Lambda_{\mu}(-pU) + \frac 1q \Lambda_\mu(qf) -\Lambda_\mu(-U)
$$
where $q=p/(p-1)$. By Donsker-Varadhan's variational formula,
$$\aligned
H(\nu|\mu_U) &=\sup_{f\in b\BB} \left(\nu(f)-\Lambda_{\mu_U}(f)\right)\\
&\ge \sup_{f\in b\BB} \left(\nu(f)- \frac 1q \Lambda_\mu(qf)\right) +\Lambda_\mu(-U)- \frac 1p \Lambda_\mu(-pU) \\
&=\frac 1q H(\nu|\mu) +\Lambda_\mu(-U)- \frac 1p \Lambda_\mu(-pU).
\endaligned
$$
Hence if $H(\nu|\mu_U)<+\infty$, $H(\nu|\mu)<+\infty$ or equivalently $\log \frac {d\nu}{d\mu} \in L^1(\nu)$,
and
$\log \frac {d\nu}{d\mu_U}= \log \frac {d\nu}{d\mu} + U + \Lambda_\mu(-U)\in L^1(\nu)$. This completes the proof of the Lemma.
\nprf

\blem\label{thmHW-lem4}  {\bf (propagation of chaos)} Let $(\nu_t)_{t\ge0}$ be the solution of the McKean-Vlasov equation with the given initial distribution $\nu_0$ such that
$\int |x|^2 d\nu_0(x)<+\infty$. Let $\mu_t^N$ be the law of $X^N(t)=(X^N_1(t),\cdots, X^N_N(t))$ solving the S.D.E. \eqref{MFS} with initial condition $\mu_0^N=\nu_0^{\otimes N}$, and $\mu_t^{N,I}$ the law of
the particles $(X^N_i(t))_{i\in I}$ for any index set $I\subset \nn^*$. Then for each $t\in \rr$ and each finite subset $I$ of $\nn^*$, $\mu_t^{N,I}\to \nu_t^{\otimes I}$ in the $L^2$-Wasserstein metric $W_2$ as $N\to\infty$.
\nlem

This is well known, see \cite{S91} or \cite{CGM08}.

\blem\label{thmHW-lem5}  {\bf (uniqueness of the minimizer of $H_W$)} If $c_{Lip,m} \|\nabla^2_{xy}W\|_\infty<1$, then the minimizer $\nu_\infty$ of the free energy $E_f(\nu)$ is unique.
\nlem

\bprf By \cite{LW18}, under (H2), if $H(\nu|\alpha)<+\infty$, $\iint W^-(x,y) d\nu(x)d\nu(y)<+\infty$ and $E_f: \MM_1(\rr^d)\to \rr$ is inf-compact. Then a minimizer $\nu_\infty$ of $E_f$ exists.

If a probability measure $\nu$ is a minimizer of $E_f$, $H(\nu|\alpha)<+\infty$, and then $\int |x|^2 d\nu<+\infty$ by (H1). Regarding the Gateaux-derivative, we see that $\nu$ must be a fixed point of the mapping $\Phi$ defined by
$$ \Phi(\nu):=\frac 1{Z'}\exp(-V -W\circledast \nu) dx
$$
where $Z'$ is the normalizing constant. Here $W\circledast \nu$ is well defined because $|W(x,y)|\le C(1+|x|^2+|y|^2)$ by the boundedness of the second derivatives of $W$.

We claim that $\Phi: \MM_1^2(\rr^d)\to \MM_1^2(\rr^d)$. Indeed, since the hamiltonian $H_\nu=V+W\circledast \nu$ (for any $\nu\in \MM_1^2(\rr^d)$) satisfies again the dissipative rate condition
$$
-\<\frac {x-y}{|x-y|}, \nabla H_\nu(x)-\nabla H_\nu(y)\>\le b_0(|x-y|), \ x,y\in\rr^d
$$
(as in \S 3), the associated generator $\LL_\nu=\Delta -\nabla H_\nu\cdot\nabla$ satisfies the Lipschitzian spectral gap estimate (\ref{lem-Lipa}) by Lemma \ref{lem-Lip}. That implies the spectral gap of $\nu'=\Phi(\nu)$, in particular $\int e^{\delta |x|}d\nu'<+\infty$ for some $\delta>0$ (\cite{BobkovLedoux-PTRF97}). Then
if $\nu\in \MM_1^2(\rr^d)$, $\Phi(\nu)\in \MM_1^2(\rr^d)$.

Now for the uniqueness of the minimizer of $E_f$, it remains to show that $\Phi$ is contractive on $(\MM_1^2(\rr^d), W_1)$. Let $\mu_k=\Phi(\nu_k), k=0,1$, and
$$
\nu_t:=(1-t)\nu_0 + t\nu_1,\ \mu_t=\Phi(\nu_t).
$$
For any $1$-Lipschitzian function $f$, we have
$$
\aligned
\frac d{dt} \mu_t(f)&= {\rm Cov}_{\mu_t}(f, -\partial_t (W\circledast \nu_t))\\
&={\rm Cov}_{\mu_t}(f, -W\circledast (\nu_1-\nu_0))
\endaligned
$$
and
$$
|\nabla_x[W\circledast (\nu_1-\nu_0)]| = |(\nabla_x W)\circledast (\nu_1-\nu_0)|\le \|\nabla^2_{yx}W\|_\infty W_1(\nu_0,\nu_1).
$$
Therefore using the Lipschitzian spectral gap estimate (\ref{lem-Lipa}) in Lemma \ref{lem-Lip} for the generator $\LL_{\nu_t}$,
$$\aligned
{\rm Cov}_{\mu_t}(f, -W\circledast (\nu_1-\nu_0))&=\<(-\LL_{\nu_t})^{-1}f, \LL_{\nu_t} W\circledast (\nu_1-\nu_0)\>_{\mu_t}\\
&=\int \<\nabla (-\LL_{\nu_t})^{-1}f, \nabla W\circledast (\nu_1-\nu_0)\> d\mu_t\\
&\le c_{Lip,m}  \|\nabla^2_{xy}W\|_\infty W_1(\nu_0,\nu_1)
\endaligned
$$
Thus we have
$$
\mu_1(f)- \mu_0(f)=\int_0^1 \frac d{dt}\mu_t(f) dt \le  c_{Lip,m}  \|\nabla^2_{xy}W\|_\infty W_1(\nu_0,\nu_1).$$
This means that $W_1(\Phi(\nu_0),\Phi(\nu_1))\le  c_{Lip,m}  \|\nabla^2_{xy}W\|_\infty W_1(\nu_0,\nu_1)$ by Kantorovitch-Rubinstein's duality relation. The proof is so completed.
\nprf

\brmk{\rm Though $(M_1^2(\rr^d), W_1)$ is not complete, the Banach's fixed point theorem works for the essential: let $\nu_\infty$ be the unique minimizer of $E_f$, then for any $\nu\in M_1^2(\rr^d)$,
$$
W_1(\Phi^n(\nu), \nu_\infty)\le [ c_{Lip,m}  \|\nabla_{xy}W\|_\infty]^n \cdot W_1(\nu,\nu_\infty), n\ge0.
$$

}\nrmk

As for the mean field relative entropy, the Fisher-Donsker-Varadhan's information $I_W(\nu)$ can be also interpreted as the mean Fisher-Donsker-Varadhan's information per particle.

\blem\label{thmHW-lem6}  {\bf (convergence of the Fisher information)} If $I(\nu|\alpha)<+\infty$,
\beqq
\frac 1N I(\nu^{\otimes N}|\mu^{(N)}) \to I_{W}(\nu).
\neqq

\nlem

\bprf For every probability measure $\nu$ on $\rr^d$ such that $I(\nu|\alpha)<+\infty$,  by the Lyapunov function condition (H1) on $V$ (\cite{GLWY}),
$$
c_1\int |x|^2 d\nu\le c_2 + I(\nu|\alpha)<+\infty.
$$
As $W$ has bounded second order derivatives, $\nabla_x W$ is of linear growth. Then $\nabla_x W\in L^2(\nu^{\otimes 2})$.
By the law of large number for i.i.d. sequence, we have
$$\aligned
\frac 1N I(\nu^{\otimes N}|\mu^{(N)})&= \frac 1{4N}  \int |\nabla \log \frac{d\nu^{\otimes N}}{d\mu^{(N)}}|^2 d\nu^{\otimes N}\\
&=\frac 1{4N} \int \sum_{i=1}^N |\nabla_{x_i} \log \frac{d\nu^{\otimes N}}{d\alpha^{\otimes N}} + \frac 1{N-1}\sum_{j\ne i}\nabla_x W(x_i,x_j) |^2  d\nu^{\otimes N}\\
&=\int \frac14 |\nabla \log \frac{d\nu}{d\alpha}(x_1) + \frac 1{N-1}\sum_{j=2}^N \nabla_x W(x_1, x_j) |^2  d\nu^{\otimes N}\\
&\to \frac14\int |\nabla \log \frac{d\nu}{d\alpha}(x_1) + \int \nabla_x W(x_1, y) d\nu(y)|^2 d\nu(x_1) =I_W(\nu).
\endaligned$$

\nprf

\subsection{Proof of Theorem \ref{thmHW}}

{\bf (1).}
At first the minimizer $\nu_\infty$ of $H_W$ is unique by Lemma ~\ref{thmHW-lem5}.

\bigskip
{\bf (2).} We may assume that $I(\nu|\alpha)<+\infty$, otherwise (\ref{H_WI_W}) is trivial for $I_W(\nu)=+\infty$. Since the Hessian $\nabla^2 V$ is lower bounded, and $V$ satisfies the Lyapunov function condition (\ref{SHJT1}), by Cattiaux-Guillin-Wu \cite{CGW10}, $\alpha$ satisfies a log-Sobolev inequality.
Then $H(\nu|\alpha)<+\infty$. By the log-Sobolev inequality of $\mu^{(N)}$ in Theorem \ref{thm2},
$$
\rho_{LS}(\mu^{(N)}) H(\nu^{\otimes N}|\mu^{(N)})\le 2 I(\nu^{\otimes N}|\mu^{(N)})
$$
and $\rho_{LS}(\mu^{(N)})\ge \rho_{LS,m}/(1-\gamma_0)^2>0$. Dividing the two sides by $N$ and letting $N$ go to infinity, we get by Lemma \ref{thmHW-lem1} and
Lemma \ref{thmHW-lem6},
$$
\rho_{LS} H_W(\nu)\le 2I_W(\nu).
$$

\bigskip
{\bf (3).} By Otto-Villani \cite{OV00} or Bobkov-Gentil-Ledoux \cite{BGL01}, the log-Sobolev inequality implies the Talagrand's $T_2$ transportation inequality, i.e.
$$
\rho_{LS}(\mu^{(N)}) W_2^2(Q,\mu^{(N)})\le 2H(Q|\mu^{(N)}), \ Q\in\MM_1((\rr^d)^N).
$$
Applying it to $Q=\nu^{\otimes N}$ with $H(\nu|\alpha)<+\infty$, we obtain
$$
\rho_{LS}(\mu^{(N)}) \frac 1N W_2^2(\nu^{\otimes N},\mu^{(N)})\le \frac 1N H(\nu^{\otimes N}|\mu^{(N)}).
$$
Notice that
$$
W_2^2(\nu^{\otimes N},\mu^{(N)})\ge \sum_{i=1}^N W_2^2(\nu, \mu^{(N,i)})=N W_2^2(\nu, \mu^{(N,1)})
$$
where $\mu^{(N,i)}$ is the marginal distribution of $x_i$ under $\mu^{(N)}$, which are all the same by the symmetry of $\mu^{(N)}$. Moreover by the uniqueness of $\nu_\infty$ and the large deviation principle of $\frac 1N \sum_{i=1}^N \delta_{x_i}$ under $\mu^{(N)}$ (\cite{LW18}), for any $f\in C_b(\rr^d)$,
$$
\mu^{(N,1)}(f)=\int \frac 1N \sum_{i=1}^N  f(x_i) d\mu^{(N)}\to \nu_\infty(f),
$$
i.e. $\mu^{(N,1)}$
converges weakly to $\nu_\infty$ . We obtain
by Lemma \ref{thmHW-lem1} and the lower semi-continuity of $W_2$,
$$
\rho_{LS} W_2^2(\nu, \nu_\infty) \le \rho_{LS} \liminf_{N\to\infty}W_2^2(\nu, \mu^{(N,1)})\le 2H_W(\nu)
$$
the desired Talagrand's type $T_2$-inequality for McKean-Vlasov equation.

\bigskip
{\bf (4).} The exponential convergence in entropy (\ref{thmHWa}) should be equivalent to the mean field log-Sobolev inequality (\ref{H_WI_W}) in part (2), basing on
\beqq\label{WHan1}
-\frac d{dt} H_W(\nu_t)=4 I_W(\nu_t)
\neqq
noted by Carrillo-McCann-Villani \cite{CMV03} in their convex framework.  The proof of (\ref{WHan1}) demands the regularity of $\nu_t$ which requires
the PDE theory of the McKean-Vlasov equation. That is why we prefer to give a rigorous probabilistic proof based directly on the log-Sobolev inequality of $\mu^{(N)}$ in Theorem \ref{thm2}.

For the exponential convergence (\ref{thmHWa}), we may and will assume that
$H_W(\nu_0)<+\infty$ and we fix the time $t>0$. By Lemma \ref{thmHW-lem1},
$$
\lim_{N\to\infty}\frac 1N H(\nu_0^{\otimes N}|\mu^{(N)})=H_W(\nu_0).
$$
Moreover by the equivalence between the log-Sobolev inequality for $\mu^{(N)}$ and the exponential convergence in entropy of the law $\mu_t^N$ of $X^N_t=(X^{N,i}_t)_{1\le i\le N}$ to $\mu^{(N)}$,
\beqq\label{Whan2}
\aligned
\frac 1N H(\mu_t^N|\mu^{(N)})&\le e^{-\rho_{LS}(\mu^{(N)}) t/2}\frac 1N H(\mu_0^{N}|\mu^{(N)})\\
&= e^{-\rho_{LS}(\mu^{(N)}) t/2}\frac 1N H(\nu_0^{\otimes N}|\mu^{(N)})<+\infty.
\endaligned\neqq
Therefore $H(\mu_t^N|\alpha^{\otimes N})<+\infty$  by Lemma \ref{thmHW-lem3}. Since $\mu_t^N$ has finite second moment (easy from the SDE theory), and $W$ has at most quadratic growth,
$$
W(x_i,x_j)\in L^1(\mu_t^N).
$$
From Lemma \ref{thmHW-lem2}, we have
$$
\frac 1N H(\mu_t^N|\alpha^{\otimes N}) \ge H(\mu_t^{N,1}|\alpha).
$$
And by the propagation of chaos (Lemma \ref{thmHW-lem4}) and the lower semi-continuity of the relative entropy $\nu\to H(\nu|\alpha)$,
$\liminf_{N\to\infty} H(\mu_t^{N,1}|\alpha)\ge H(\nu_t|\alpha)$.

 So we get
$$\aligned
\liminf_{N\to\infty}\frac 1N H(\mu_t^N|\mu^{(N)})&=\liminf_{N\to\infty}\left(\frac 1N H(\mu_t^N|\alpha^{\otimes N}) + \int \frac 1{N(N-1)}\sum_{1\le i<j\le N}W(x_i,x_j) d\mu_t^N + \frac 1N \log \tilde Z_N\right)\\
&\ge H(\nu_t|\alpha) + \liminf_{N\to\infty}\frac 12 \int W(x_1,x_2) d\mu_t^N - \inf_{\nu\in\MM_1(\rr^d)} E_f(\nu)\\
&=H(\nu_t|\alpha) + \frac 12 \iint W(x_1,x_2) d\nu_t(x_1) d\nu_t(x_2) - \inf_{\nu\in\MM_1(\rr^d)} E_f(\nu)\\
&=H_W(\nu_t)
\endaligned
$$
by the $W_2$-propagation of chaos in Lemma \ref{thmHW-lem4}. Plugging it into (\ref{Whan2}), we obtain the exponential convergence in entropy (\ref{thmHWa}). That implies the $W_2$-exponential convergence  (\ref{thmHWbb}) by Talagrand's type $T_2$-inequality (\ref{T2}). \hfill$\Box$

\bigskip
{\bf Acknowledgements: } W. Liu is supported by the NSFC 11731009. Part of these results were first presented in the ``Workshop on stability of functional inequalities and applications'' in 2018 in Toulouse which is supported by the Labex CIMI and the ANR project ``Entropies, Flots, Inegalites”.


\bibliographystyle{plain}
\bibliography{references}

\end{document}